\documentclass[11pt]{article}

\usepackage[T1]{fontenc}
\usepackage[utf8]{inputenc}
\usepackage[margin=0.99in]{geometry}
\usepackage{amsmath,amssymb,amsfonts,amsthm,mathtools}
\usepackage{mathrsfs}
\usepackage{enumitem}
\usepackage[colorlinks=true,linkcolor=blue,citecolor=blue,urlcolor=blue]{hyperref}

\def\0{{\bf 0}}

\def\d{{\mathrm d}}

\def\p{{\mathbf p}}
\def\q{{\mathbf q}}
\def\s{{\mathbf s}}

\def\R{{\mathbb R}}

\def\u{{\bf u}}

\def\x{{\bf x}}
\def\y{{\bf y}}

\newcommand{\Lie}{\mathcal L}

\newcommand{\Om}{\Omega}
\newcommand{\lam}{\lambda}

\newcommand{\mcC}{\mathcal C}

\newtheorem{theorem}{Theorem}[section]
\newtheorem{proposition}[theorem]{Proposition}
\newtheorem{lemma}[theorem]{Lemma}

\theoremstyle{definition}
\newtheorem{definition}[theorem]{Definition}
\newtheoremstyle{italicremark}%
  {\topsep}{\topsep}
  {\itshape}
  {}
  {\itshape}
  {.}
  { }
  {}
\theoremstyle{italicremark}
\newtheorem{remark}[theorem]{Remark}

\begin{document}

\title{A Note on Singular Boundary Regularisation in Hamiltonian Systems}

\author{Cristina Stoica\footnote{Email: cstoica@wlu.ca}\\
Wilfrid Laurier University, Canada}

\maketitle

\medskip
\noindent
\textbf{Abstract:} We consider the $N$-body problem with a homogeneous potential of degree $-\alpha$. The McGehee blow-up of a collision singularity extends the vector field to a fictitious boundary, \textit{the collision manifold}. The price is the symplectic structure: the pulled-back canonical form is singular there. 
We ask whether a conformal factor depending only on the boundary defining function can make this form extend to the boundary as a symplectic, or more generally a $b^m$-symplectic, form. The answer is negative, with two exceptions: $\alpha=2$, where the canonical form is $b$-symplectic, and collisions in which the colliding part has a single degree of freedom.


\section{Introduction}\label{sec:intro}

In celestial mechanics, McGehee's blow-up \cite{McGehee1974,McGehee1981} replaces a collision singularity of the $N$-body problem by a boundary, \textit{the collision manifold}, on which the rescaled vector field extends smoothly.  In this way, the flow on the collision manifold, together with continuity with respect to initial data, provides information on the motion near collision.  The rescaling of the momenta and of time is not canonical, so it does not preserve the Hamiltonian character of the equations.  This note asks what is left of the canonical symplectic form at the boundary.

For a total collision two descriptions are available.  The scaling symmetry of a homogeneous potential reduces the problem to a contact system on the collision manifold \cite{Bryant1980,LacombaIbort1988,BravettiJackmanSloan2023}.  On the other hand, singular symplectic structures, in particular $b^m$-symplectic structures \cite{GuilleminMirandaPires2014,Scott2016}, have been used to describe the geometry at infinity and at collision  \cite{DelshamsKiesenhoferMiranda2017,BraddellEtAl2019,MirandaPlanas2022}, after suitable blow-ups. 
The question we address is whether such a description exists for McGehee's blow-up: given the canonical form $\omega$ on the interior, is there a function $f$, depending only on the distance to the boundary, such that $f\omega$ extends to the boundary as a nondegenerate, possibly $b^m$-symplectic, form?

The answer rests on a scaling argument stated in Proposition \ref{prop:scaling-obstruction}.  If $\omega$ has an expansion in powers of the boundary defining function whose leading part is degenerate, then no conformal factor makes $\omega$ extend as a nondegenerate form.  Passing to $b^m$-forms changes the weights of the components containing the normal direction only, and this leads to a distinction, Definition~\ref{def:isotropic}, between isotropic boundaries, at which the tangential part of $\omega$ has a single weight, and anisotropic ones, at which it has several.  At an anisotropic boundary no conformal rescaling of $\omega$ extends as a nondegenerate $b^m$-form for any $m$; see Proposition~\ref{prop:bm-obstruction}.  The hypotheses of both propositions are independent of the choice of collar.

We apply Proposition \ref{prop:scaling-obstruction} to three classes of singularities.  For a central force problem with potential homogeneous of degree $-\alpha$, McGehee's boundary is isotropic and $\omega=r^{\beta}\widetilde\omega$, where $\beta=1-\alpha/2$ and $\widetilde\omega$ is a nondegenerate $b$-form, closed if and only if $\beta=0$.  For $\beta\neq0$ the boundary carries the contact structure of \cite{LacombaIbort1988,BravettiJackmanSloan2023}, and no conformal rescaling of $\omega$ is $b^m$-symplectic, closedness being the obstruction; for $\beta=0$ the form is $b$-symplectic.  The same holds for the total collision of the planar three-body problem in the reduced and regularised form of \cite{MoeckelMontgomery2013}, at every point of the collision manifold, including those over the regularised binary collisions; to our knowledge this has not been noted before.  For a partial collision, in which part of the system stays regular, the boundary is anisotropic as soon as the colliding part has at least two degrees of freedom and $\beta\neq0$, and no $b^m$-symplectic rescaling exists.  Simultaneous binary collisions are the main example.  Their regularisation and the dynamics near them have been studied in \cite{SimoLacomba1992,ElBialy1993,MartinezSimo1999,MartinezSimo2000,Bakker2010,Bakker2011,Bakker2012,DuignanDullin2021,DuignanQuaschner2026}, but the fate of the symplectic form at the blown-up boundary is not considered in these works.  
In both the total and the partial case the result depends on the choice of blow-up variables: a different rescaling of the momenta makes the canonical form $b$-symplectic for every $\alpha>0$, but for $\alpha\neq2$ the resulting boundary carries no collision manifold.

\medskip
The note is organised as follows.  Section~\ref{sec:time-changes} collects facts on time changes and conformal changes of a symplectic form.  Section~\ref{sec:criterion} contains the scaling criterion, the notion of isotropic and anisotropic boundaries, and the obstruction for $b^m$-forms.  Section~\ref{sec:mcgehee} treats central force problems, the total collision of the planar three-body problem and partial collisions.  We end with some concluding remarks.

\section{Time changes and conformal symplecticity}\label{sec:time-changes}

In this section we collect some classical facts; see, e.g., \cite{AbrahamMarsden1978,Arnold1989} for background.  Let $(M,\omega)$ be a symplectic manifold and let $H:M\to\R$ be a smooth Hamiltonian.  The Hamiltonian vector field $X_H$ is determined by $ i_{X_H}\omega=-\d H.$
Let $g:M\to(0,\infty)$ be a smooth positive function and define $  Y=gX_H.$
The vector fields $X_H$ and $Y$ have the same oriented trajectories, but their parametrisations differ.  
The following elementary fact is often useful.

\begin{proposition}\label{prop:time-change-ham}
Let $(M,\omega)$ be symplectic and let $Y=gX_H$, with $g>0$.  Then $Y$ is locally Hamiltonian with respect to the same symplectic form $\omega$ if and only if \[ \d g\wedge \d H=0.\]
In particular, on a region where the regular level sets of $H$ are connected, this condition implies that $g$ is locally a function of $H$ alone.  In that case there exists a function $K=F(H)$ such that
\[
        X_K=gX_H.
\]
\end{proposition}

\begin{proof}
Since
$\displaystyle{
        i_Y\omega=i_{gX_H}\omega=-g\,\d H,}$
the vector field $Y$ is locally Hamiltonian with respect to $\omega$ if and only if $g\,\d H$ is closed.  But
$\displaystyle{
        \d(g\,\d H)=\d g\wedge \d H.}
$
This proves the first claim.  If $\d H\neq0$ and $\d g\wedge\d H=0$, then $g$ is locally constant along the regular level sets of $H$.  If these level sets are connected, $g$ is locally a function of $H$.  Conversely, if $g=g(H)$, then $K=F(H)$, with $F'=g$, satisfies $\displaystyle{\d K=g(H)\,\d H}$
so $X_K=gX_H$.
\end{proof}

Thus a general time change preserves the orbit foliation but not necessarily the Hamiltonian structure.  A related calculation concerns conformal changes of the symplectic form.  Formally,
\[
        i_{gX_H}\left(\frac{1}{g}\omega\right)=-\d H.
\]
However, $\displaystyle{\frac{1}{g}\,\omega}$ is  not symplectic in general.

\begin{proposition}\label{prop:conformal-symplectic}
Let $(M,\omega)$ be a symplectic manifold of dimension $2n\geq4$.  If $g:M\to(0,\infty)$ is smooth, then $\displaystyle{\frac{1}{g}\,\omega}$ is closed if and only if $g$ is locally constant.
\end{proposition}

\begin{proof}
Since $\d\omega=0$,
\[
        \d \left(\frac{1}{g}\,\omega \right)=-\frac{1}{g^2} \d g\wedge\omega.
\]
Thus $\displaystyle{\frac{1}{g}\,\omega}$ is closed if and only if $\d g\wedge\omega=0$.  Wedging with $\omega^{n-2}$ gives
\[
        \d g\wedge\omega^{n-1}=0.
\]
The map $\alpha\mapsto \alpha\wedge\omega^{n-1}$ is an isomorphism from one-forms to $(2n-1)$-forms.  Hence $\d g=0$.
\end{proof}

\begin{remark}
The two-dimensional case is exceptional: every three-form vanishes, so multiplying an area form by a nonzero function again gives a closed two-form.  This low-dimensional fact does not persist in higher-dimensional Hamiltonian systems.
\end{remark}

\section{A finite-dimensional scaling criterion}\label{sec:criterion}

The preceding discussion assumes that $g$ is smooth and positive.  In singular regularisation, the factor $g$ may vanish or blow up at a boundary.  Then the rescaled vector field may extend smoothly to the boundary while the symplectic form, or its inverse Poisson tensor, becomes singular or degenerate.

\medskip

Let $\lam : \widetilde M \to [0, \infty)$  be a boundary defining function on a manifold $\widetilde M$, that is $\lam $ is smooth, vanishes on the boundary $\partial\widetilde M$ and $d\lam\neq 0$  at every point of $\partial\widetilde M.$
Further, let $X$ be a vector field near $\partial\widetilde M$ with $\d\lam(X)=1$, and let $\lam_0>0$.  The map $(p,t)\mapsto\varphi_t(p)$, where $\varphi_t$ is the flow of $X$, identifies $\partial\widetilde M\times[0,\lam_0)$ with the neighbourhood $\{\lam<\lam_0\}$ of the boundary, and $\lam(\varphi_t(p))=t$.

\begin{definition}\label{def:admissible}
Fix $\lam_0>0$.  A function $f:(0,\lam_0)\to\R\setminus\{0\}$ is called an \emph{admissible conformal factor} if there are $\tilde c\in\R$ and $g$ smooth and nonvanishing on $[0,\lam_0)$ such that $f$ has the form
\[
        f(\lam)=\frac{1}{\lam^{\tilde c}}\,g(\lam)\,.
\]
 A two-form $\omega$ on the interior $\{0<\lam<\lam_0\}$ is said to be \emph{conformally regularisable at $\lam=0$} if there exists an admissible conformal factor $f$ such that $f(\lam)\,\omega$ extends smoothly to $\lam=0$ as a two-form that is non-degenerate at every point of the boundary.
\end{definition}

\begin{proposition}[Scaling obstruction]\label{prop:scaling-obstruction}
Let $\omega$ be a symplectic form on the interior of a manifold $\widetilde M$ with boundary, and suppose that, in some collar neighbourhood of the boundary $\partial\widetilde M$, it has an expansion
\[
        \omega
        =\sum_{j=1}^J \lam^{a_j}\omega_j
        +\d\lam\wedge\sum_{\ell=1}^L \lam^{b_\ell}\eta_\ell,
        \qquad J,L\in\{1,2,\dots\}\cup\{\infty\},
\]
where:
\begin{enumerate}[label=(\alph*)]
\item the exponents satisfy $a_1<a_2<\cdots$ and $b_1<b_2<\cdots$;
\item the two-forms $\omega_j$ and one-forms $\eta_\ell$ are smooth up to the boundary, independent of $\lam$, do not involve $\d\lam$, and none vanishes identically along $\partial\widetilde M$;

\item if $J=\infty$ or $L=\infty$, the corresponding exponents tend to infinity, and the expansion is asymptotic, that is for every $K\in\R$, the difference between $\omega$ and the sum of its terms with exponent less than $K$ is $O(\lam^K)$ as $\lam\to0$, locally uniformly on $\partial\widetilde M$.
\end{enumerate}
Let $c=\min\{a_1,b_1\}$ be the minimal exponent, and let
\[
        \omega_0^{\mathrm{lead}}=
        \begin{cases}
        \omega_1 & \text{if } a_1<b_1,\\
        \d\lam\wedge\eta_1 & \text{if } b_1<a_1,\\
        \omega_1+\d\lam\wedge\eta_1 & \text{if } a_1=b_1,
        \end{cases}
\]
be the leading weighted part, that is, the sum of the terms with exponent $c$.  If $\omega_0^{\mathrm{lead}}$ is degenerate at some point of the boundary $\partial\widetilde M$, then $\omega$ is not conformally regularisable at $\lam=0$ in the sense of Definition~\ref{def:admissible}.
\end{proposition}

\begin{proof}
Let $\displaystyle{f(\lam)=\frac{1}{\lam^{\tilde c}}g(\lam)}$ be an admissible conformal factor.

\medskip
\noindent
\emph{Case 1: $J,L<\infty$\,.}  In this case we have
\[
        f(\lam)\,\omega
        =g(\lam)\left(\sum_{j=1}^J \lam^{a_j-\tilde c}\omega_j
        +\d\lam\wedge\sum_{\ell=1}^L \lam^{b_\ell-\tilde c}\eta_\ell\right).
\]
Since the forms $\omega_j$ do not involve $\d\lam$, there is no cancellation between the two groups of terms, and since the exponents within each group are distinct, there is no cancellation within a group.

\smallskip
\noindent
If $\tilde c>c$, the component with exponent $c-\tilde c<0$ is unbounded as $\lam\to0$ at points of the boundary where the corresponding form does not vanish, so $f(\lam)\,\omega$ does not extend smoothly to $\lam=0$.  If $\tilde c<c$, every component tends to zero, and the limit is the zero two-form, which is degenerate.  Hence the only admissible choice yielding a finite nonzero limit is $\tilde c=c$, in which case
\[
        \lim_{\lam\to0}f(\lam)\,\omega
        =g(0)\,\omega_0^{\mathrm{lead}}.
\]
By hypothesis this two-form is degenerate at some point of the boundary.  Thus no admissible conformal factor produces a smooth limiting two-form that is non-degenerate at every boundary point, and $\omega$ is not conformally regularisable.

\medskip
\noindent
\emph{Case 2: $J=\infty$ or $L=\infty$\,.}  By (a) and (c), the exponents greater than $c$ have a smallest element, which we write $c+\delta$ with $\delta>0$.  The only terms with exponent less than $c+\delta$ are those with exponent $c,$  so (c) with $K=c+\delta$ gives
\[
        \omega=\lam^{c}\,\omega_0^{\mathrm{lead}}+O\left(\lam^{c+\delta}\right)
\]
locally uniformly on $\partial\widetilde M$, and therefore
\[
        f(\lam)\,\omega=g(\lam)\,\lam^{c-\tilde c}\left(\omega_0^{\mathrm{lead}}+O\left(\lam^{\delta}\right)\right).
\]
The two-form $\omega_0^{\mathrm{lead}}$ is not identically zero on $\partial\widetilde M$, by (b), since its component without $\d\lam$ and its component containing $\d\lam$ cannot cancel.  If $\tilde c>c$, the right-hand side is unbounded near the points of $\partial\widetilde M$ where $\omega_0^{\mathrm{lead}}\neq0$.  If $\tilde c<c$, it tends to the zero two-form.  If $\tilde c=c$, it tends to $g(0)\,\omega_0^{\mathrm{lead}}$, which is degenerate at some point of the boundary.  As in Case 1, $\omega$ is not conformally regularisable. \end{proof}

\begin{remark}\label{rem:collar} 
Allowing infinite expansions makes the hypotheses independent of the collar.
Another vector field $X'$ with $\d\lam(X')=1$ gives another collar. 
The two collars agree on $\partial\widetilde M$ and on the value of $\lam$, and 
they differ only in which point of a level set $\{\lam=\text{const}\}$ lies above a given point of $\partial\widetilde M$. 
Passing from one collar to the other therefore moves the points of each level set by a diffeomorphism of $\partial\widetilde M$, which depends smoothly on $\lam$ and is the identity for $\lam=0$.
 Taylor expanding this dependence in $\lam$ turns, in general, a finite expansion into an infinite one, as in Case~2, with the same minimal exponent $c$, and with a leading part which at each boundary point $p$ is the image of $\omega_0^{\mathrm{lead}}(p)$ under a linear isomorphism of $T_p\widetilde M$. 
 \end{remark}

\bigskip

The proposition above shows that a single conformal factor cannot make the original symplectic structure extend as an ordinary non-degenerate symplectic form when the leading weighted part is degenerate. However,  it is possible that  one obtains a degenerate Poisson structure, a singular symplectic structure, or a presymplectic structure on the boundary.

We recall the singular structures relevant here; see \cite{GuilleminMirandaPires2014,Scott2016}.  For an integer $m\geq1$, the $b^m$-tangent bundle ${}^{b^m}T\widetilde M$ is the vector bundle whose sections are the vector fields $V$ on $\widetilde M$ such that $V(\lam)=\lam^m h$ for some smooth function $h$.  In a collar it is spanned by $\lam^m X$ and the vector fields tangent to the level sets of $\lam$, and its dual is spanned by $\d\lam/\lam^m$ and the one-forms on the level sets.  A $b^m$-form is a section of $\Lambda^k\left({}^{b^m}T^*\widetilde M\right)$, and $b^m$-symplectic forms are the closed $b^m$-two-forms that are nondegenerate at every point, including the boundary.  We set ${}^{b^0}T\widetilde M=T\widetilde M$, so that $b^0$-forms are ordinary forms.

Passing from ordinary forms to $b^m$-forms changes the exponents of the components containing $\d\lam$ only.  Since $\d\lam=\lam^m\left(\d\lam/\lam^m\right)$, each exponent $b_\ell$ in Proposition~\ref{prop:scaling-obstruction} becomes $b_\ell+m$, while the exponents $a_j$ are unchanged.  We therefore separate the two kinds of components.  Near $\partial\widetilde M$ the distribution $\ker\d\lam$ is tangent to the level sets of $\lam$, and along the boundary it coincides with $T\partial\widetilde M$.  The \emph{tangential part} of $\omega$ is its restriction $\sigma$ to $\ker\d\lam$; it is defined on the interior near the boundary and does not depend on a collar.  In a collar in which $\omega$ has the expansion of Proposition~\ref{prop:scaling-obstruction},
\[
        \sigma=\sum_{j=1}^J\lam^{a_j}\omega_j,
\]
since the terms containing $\d\lam$ vanish on $\ker\d\lam$.  We call the components $\lam^{a_j}\omega_j$ tangential and the components $\lam^{b_\ell}\,\d\lam\wedge\eta_\ell$ normal.

\begin{definition}\label{def:isotropic}
Let $\dim\widetilde M=2n$ with $n\geq2$.  Suppose that for some $c\in\R$ the form $\lam^{-c}\sigma$ extends continuously to $\partial\widetilde M$, with boundary value $\sigma_0$ that is a section of $\Lambda^2T^*\partial\widetilde M$ that is not identically zero.  We call $\sigma_0$ the \emph{leading tangential part} of $\omega$.  We say that $\omega$ is \emph{isotropic} at $\partial\widetilde M$ if $\sigma_0$ has rank $2n-2$ at every point of $\partial\widetilde M$, and \emph{anisotropic} at $\partial\widetilde M$ if $\sigma_0$ has rank less than $2n-2$ at some point of $\partial\widetilde M$.
If no such $c$ exists, $\omega$ is neither isotropic nor anisotropic at $\partial\widetilde M$.
\end{definition}

The rank $2n-2$ is the largest possible, since $T_p\partial\widetilde M$ has dimension $2n-1$; it is also the rank of $\sigma$ at every interior point, because $\omega$ is nondegenerate there and the restriction of a nondegenerate two-form to a hyperplane has rank $2n-2$.  If $\omega$ has the expansion of Proposition~\ref{prop:scaling-obstruction}, then $c=a_1$ and
\[
        \lam^{-c}\sigma=\omega_1+\sum_{j\geq2}\lam^{a_j-a_1}\omega_j ,
\]
an identity if $J<\infty$, and otherwise an asymptotic expansion in the sense of (c) in Proposition~\ref{prop:scaling-obstruction}.  Since $a_j-a_1>0$ for $j\geq2$, in both cases $\sigma_0=\omega_1|_{T\partial\widetilde M}$, which is not identically zero by (b).  If there is a single tangential exponent, that is $J=1$, then $\omega_1=\lam^{-a_1}\sigma$ has rank $2n-2$ at interior points and, being independent of $\lam$ in the collar, also along the boundary.  Thus, for forms with such an expansion, anisotropy requires at least two distinct tangential exponents, with a degenerate leading tangential part.

\begin{proposition}[Anisotropic obstruction for $b^m$-forms]\label{prop:bm-obstruction}
Let $\dim\widetilde M=2n$ with $n\geq2$, and suppose that $\omega$ is anisotropic at $\partial\widetilde M$ in the sense of Definition~\ref{def:isotropic}.  Then for every integer $m\geq0$ there is no admissible conformal factor $f$ such that $f\omega$, $\lam>0$, extends to $\widetilde M$ as a smooth $b^m$-form that is nondegenerate at every point of $\partial\widetilde M$.
\end{proposition}

\begin{proof}
Let $\displaystyle{f(\lam)=\frac{1}{\lam^{\tilde c}}g(\lam)}$ be an admissible conformal factor, and suppose that $f\omega$ extends to a smooth $b^m$-form $\Om$ that is nondegenerate along $\partial\widetilde M$.  The vector fields tangent to the level sets of $\lam$ are sections of ${}^{b^m}T\widetilde M$, so $\ker\d\lam$ is a subbundle of ${}^{b^m}T\widetilde M$ of corank one.  Hence the restriction of $\Om$ to $\ker\d\lam$, which equals $f\sigma$ on the interior, extends smoothly to the boundary, and at every boundary point it has rank $2n-2$, being the restriction of a nondegenerate two-form to a hyperplane.  Now
\[
        f\sigma=\lam^{c-\tilde c}\,g(\lam)\left(\lam^{-c}\sigma\right),
\]
where, by Definition~\ref{def:isotropic}, $\lam^{-c}\sigma$ tends to $\sigma_0$ at the boundary, and $\sigma_0$ is not identically zero.
If $\tilde c>c$, then $f\sigma$ is unbounded near the points of $\partial\widetilde M$ where $\sigma_0\neq0$.  If $\tilde c<c$, it tends to zero at the boundary, so its boundary value has rank $0<2n-2$.  If $\tilde c=c$, its boundary value is $g(0)\,\sigma_0$, whose rank is less than $2n-2$ at some boundary point.  Each case contradicts the properties of $\Om$.
\end{proof}

\section{Collision singularities}\label{sec:mcgehee}

\subsection{Central force problems}\label{subsec:central}

Let $n\geq2$, let $\mathcal O$ be an open subset of the unit sphere ${\mathcal S}^{n-1}\subset\R^n$, and write $\q=r\s$ with $r>0$ and $\s\in\mathcal O$.  Let $(r,p_r, \s, \p_\s)$ be the induced canonical coordinates on $T^*\{r\s:\ r>0,\ \s\in\mathcal O\}\cong T^*(0,\infty)\times T^*\mathcal O$, so that the canonical one-form and the symplectic form are
\[
        \theta=p_r\,\d r+\p_\s\cdot\d\s,
        \qquad
        \omega=\d\theta.
\]
Here, in a canonical chart $(s^i,p_i)$ of $T^* {\mathcal S}^{n-1}$, $\p_\s\cdot\d\s=\sum_ip_i\,\d s^i$ is the canonical one-form of $T^* {\mathcal S}^{n-1}$ and $|\p_\s|$ is the norm given by the round metric on ${\mathcal S}^{n-1}$.
 Consider  a Hamiltonian of the form
\begin{equation}\label{eq:H}
        H=\frac12\left(p_r^2+\frac{|\p_\s|^2}{r^{2}}\right)-r^{-\alpha}U(\s),
        \qquad \alpha>0,
\end{equation}
with the potential $U$ smooth and positive on $\mathcal O$.  For $\mathcal O={\mathcal S}^{n-1}$ and $U$ constant this is the central force problem, whereas for non-constant $U$ it includes, for instance, the anisotropic Kepler problem \cite{Devaney1978}.  For the total collision of the $N$-body problem, $\q\in\R^n$ collects the mass-weighted Jacobi coordinates, $U$ is the restriction of the potential to the unit sphere, and $\mathcal O$ is the complement of the configurations in which a proper subset of the bodies collides.  Note that for Newtonian gravity $\alpha=1$.  McGehee's variables and time \cite{McGehee1981} are
\begin{equation}\label{eq:mcg}
        v=r^{\alpha/2}p_r,
        \qquad
        \u=r^{\alpha/2-1}\p_\s,
        \qquad
        \d t=r^{1+\alpha/2}\,\d\tau .
\end{equation}
Then $H=r^{-\alpha}\left[\frac12\left(v^2+|\u|^2\right)-U(\s)\right]$.  The variables $(r,v,\s,\u)$ are coordinates on $(0,\infty)\times\R\times T^*\mathcal O$, and we call $\widetilde M=[0,\infty)\times\R\times T^*\mathcal O$ the blown-up space.  Below, canonical charts and the canonical form $\d\u\wedge\d\s$ of $T^*{\mathcal S}^{n-1}$ are understood on the open subset $T^*\mathcal O$.  On the energy level $H=h$ we have $\frac12\left(v^2+|\u|^2\right)-U(\s)=hr^\alpha$, and the boundary $r=0$ of $\widetilde M$ contains the collision manifold
\[
        \mcC=\left\{r=0,\ \tfrac12\left(v^2+|\u|^2\right)=U(\s)\right\}.
\]
For $r>0$ the time change has $g=r^{1+\alpha/2}$, and $\d g\wedge\d H$ is not identically zero, so by Proposition~\ref{prop:time-change-ham} the $\tau$-flow is not locally Hamiltonian for the canonical form $\omega$.  Let us denote:
\begin{equation}
\beta=\frac{2-\alpha}{2}\,.\end{equation}
Also, for future reference, in a canonical chart of $T^*{\mathcal S}^{n-1}$ we write $\partial_\s$ and $\partial_\u$ for the partial derivatives with respect to the base and fibre coordinates.

\begin{lemma}\label{lem:mcgehee-equations}
In the variables $(r,v,\s,\u)$ and the time $\tau$, Hamilton's equations for \eqref{eq:H} read
\begin{equation}\label{eq:blownup}
\begin{aligned}
        r'&=r\,v, \\
        v'&=(1-\beta)\,v^2+|\u|^2-2(1-\beta)\,U(\s), \\
        \s'&=\partial_\u\left(\tfrac12|\u|^2\right), \\
        \u'&=-\partial_\s\left(\tfrac12|\u|^2-U(\s)\right)-\beta\,v\,\u ,
\end{aligned}
\end{equation}
where prime denotes differentiation with respect to $\tau$.  Along these equations
\[
        \frac{\d}{\d\tau}\left[\frac12\left(v^2+|\u|^2\right)-U(\s)\right]
        =2(1-\beta)\,v\left[\frac12\left(v^2+|\u|^2\right)-U(\s)\right],
\]
so $\mcC$ is invariant, and on $\mcC$
\begin{equation}\label{eq:vprime}
        v'=\beta\,|\u|^2 .
\end{equation}
\end{lemma}

\begin{proof}
Hamilton's equations for \eqref{eq:H} are
\[
\begin{aligned}
        &\dot r=p_r,\qquad
        &&\dot p_r=\frac{|\p_\s|^2}{r^3}-2(1-\beta)\, r^{2\beta-3}U(\s),\\
        &\dot\s=\frac{1}{r^2}\,\partial_{\p_\s}\left(\tfrac12|\p_\s|^2\right),\qquad
        &&\dot \p_\s=-\frac{1}{r^2}\,\partial_\s\left(\tfrac12|\p_\s|^2\right)+r^{2\beta-2}\partial_\s U(\s).
\end{aligned}
\]
Since $\frac12|\p_\s|^2$ is quadratic in $\p_\s$, the substitution $\p_\s=r^{\beta}\u$ gives $\partial_{\p_\s}\left(\tfrac12|\p_\s|^2\right)=r^{\beta}\,\partial_\u\left(\tfrac12|\u|^2\right)$ and $\partial_\s\left(\tfrac12|\p_\s|^2\right)=r^{2\beta}\,\partial_\s\left(\tfrac12|\u|^2\right)$.  Differentiating $v=r^{1-\beta}p_r$ and $\u=r^{-\beta}\p_\s$ with respect to $t$, multiplying by $r^{2-\beta}=\d t/\d\tau$, and using $p_r=r^{\beta-1}v$, we obtain
\[
\begin{aligned}
        r'&=r^{2-\beta}\,p_r=r\,v,
        \qquad         \qquad        \qquad   \,
        \s'=r^{2-\beta}\,r^{-2}\,r^{\beta}\,\partial_\u\left(\tfrac12|\u|^2\right)=\partial_\u\left(\tfrac12|\u|^2\right),\\
        v'&=(1-\beta)\,r^{2-2\beta}p_r^2+r^{3-2\beta}\,\dot p_r
        =(1-\beta)\,v^2+r^{-2\beta}|\p_\s|^2-2(1-\beta)\,U(\s),\\
        \u'&=-\beta\,r^{1-2\beta}\,p_r\,\p_\s+r^{2-2\beta}\,\dot \p_\s
        =-\beta\,v\,\u-\partial_\s\left(\tfrac12|\u|^2-U(\s)\right),
\end{aligned}
\]
since $r^{-2\beta}|\p_\s|^2=|\u|^2$ and $r^{1-2\beta}p_r\,\p_\s=v\,\u$.
For the invariance, the last two equations give $\frac{\d}{\d\tau}\left(\tfrac12|\u|^2-U(\s)\right)=-\beta\,v\,|\u|^2$, since the terms involving $\partial_\s$ cancel and $\partial_\u\left(\tfrac12|\u|^2\right)\cdot\u=|\u|^2$.  Adding $v\,v'$ gives the displayed identity.  Since also $r'=rv$, the set $\mcC$ is invariant.  On $\mcC$, $2(1-\beta)\,U(\s)=(1-\beta)\left(v^2+|\u|^2\right)$, which gives \eqref{eq:vprime}.
\end{proof}

The canonical one-form $\theta=p_r\,\d r+\p_\s\cdot\d\s$ becomes $\theta=r^{\beta-1}v\,\d r+r^{\beta}\u\cdot\d\s$, and
\begin{equation}\label{eq:omega-tilde}
        \omega=\d\theta=r^{\beta}\,\widetilde\omega,
        \qquad
        \widetilde\omega=\frac{\d r}{r}\wedge\left(\beta\,\u\cdot\d\s-\d v\right)+\d\u\wedge\d\s .
\end{equation}
The collar is given by the vector field $X=\partial_r$ in the variables $(r,v,\s,\u)$.  Its integral curves are the orbits of the scaling generated by $r\partial_r$, under which $r$ is multiplied by a positive constant and $v$, $\s$, $\u$ are unchanged.  For $n=1$ the angular variables are absent and $r^{1-\beta}\omega=\d v\wedge\d r$ is nondegenerate.

\begin{proposition}\label{prop:mcgehee-isotropic}
Let $n\geq2$ and $\beta<1$.
\begin{enumerate}[label=(\alph*)]
\item The form $\widetilde\omega$ extends to the blown-up space as a $b$-form which is nondegenerate at every point of the boundary $r=0$, and
\[
        \d\widetilde\omega=-\beta\,\frac{\d r}{r}\wedge\widetilde\omega ,
\]
so that $\widetilde\omega$ is closed if and only if $\beta=0$.
\item The form $\omega$ is not conformally regularisable at $r=0$ in the sense of Definition~\ref{def:admissible}.
\item The form $\omega$ is isotropic at $r=0$ in the sense of Definition~\ref{def:isotropic}.
\item If $\beta\neq0$, then for every integer $m\geq1$ there is no admissible conformal factor $f$ such that $f\omega$ extends to the blown-up space as a $b^m$-symplectic form.
\end{enumerate}
\end{proposition}

\begin{proof}
(a) In the $b$-coframe $\d r/r$, $\d v$, $\d\s$, $\d\u$ the coefficients of $\widetilde\omega$ are smooth and independent of $r$, so $\widetilde\omega$ extends to $r=0$ as a $b$-form.  Write $\widetilde\omega=A+B$ with $A=\frac{\d r}{r}\wedge\left(\beta\,\u\cdot\d\s-\d v\right)$ and $B=\d\u\wedge\d\s$.  Since $A\wedge A=0$, $B^n=0$ and $(\u\cdot\d\s)\wedge B^{n-1}=0$, because $B$ is the canonical form of $T^*{\mathcal S}^{n-1}$,
\[
        \widetilde\omega^{\,n}=n\,A\wedge B^{n-1}=-n\,\frac{\d r}{r}\wedge\d v\wedge\left(\d\u\wedge\d\s\right)^{n-1},
\]
a nowhere vanishing section of $\Lambda^{2n}\left({}^{b}T^*\widetilde M\right)$.  Since $\d(\d r/r)=0$,
\[
        \d\widetilde\omega=-\frac{\d r}{r}\wedge\d\left(\beta\,\u\cdot\d\s\right)
        =-\beta\,\frac{\d r}{r}\wedge\d\u\wedge\d\s
        =-\beta\,\frac{\d r}{r}\wedge\widetilde\omega ,
\]
and $\frac{\d r}{r}\wedge\d\u\wedge\d\s\neq0$.

(b) In the collar given by $X$, \eqref{eq:omega-tilde} reads
\[
        \omega=r^{\beta}\,\d\u\wedge\d\s+\d r\wedge r^{\beta-1}\left(\beta\,\u\cdot\d\s-\d v\right),
\]
which is the expansion of Proposition~\ref{prop:scaling-obstruction} with $\lam=r$, $a_1=\beta$, $\omega_1=\d\u\wedge\d\s$, $b_1=\beta-1$ and $\eta_1=\beta\,\u\cdot\d\s-\d v$.  The minimal exponent is $\beta-1$ and the leading part is $\d r\wedge\eta_1$, which has rank two at every boundary point, since $\eta_1$ contains $-\d v$.  As $2n\geq4$, it is degenerate, and Proposition~\ref{prop:scaling-obstruction} applies.

(c) The terms containing $\d r$ vanish on $\ker\d r$, so the tangential part of $\omega$ is $\sigma=r^{\beta}\,\d\u\wedge\d\s$.  Hence $r^{-\beta}\sigma=\d\u\wedge\d\s$, which has rank $2n-2$ on the tangent space of the boundary, spanned by the $v$-, $\s$- and $\u$-directions.

(d) Suppose that $f\omega$ extends as a $b^m$-symplectic form.  Then $f\omega$ is closed on $\{0<r<\lam_0\}$, which is connected, so $f$ is a constant $c\neq0$ by Proposition~\ref{prop:conformal-symplectic}, applied to $1/|f|$.  In the $b^m$-coframe $\d r/r^m$, $\d v$, $\d\s$, $\d\u$, the coefficient of $\d\u\wedge\d\s$ in $c\,\omega$ is $c\,r^{\beta}$, which is bounded near $r=0$ only if $\beta\geq0$.  By the computation in (a),
\[
        (c\,\omega)^n=c^n r^{n\beta}\,\widetilde\omega^{\,n}=-n\,c^n\,r^{\,n\beta+m-1}\,\frac{\d r}{r^m}\wedge\d v\wedge\left(\d\u\wedge\d\s\right)^{n-1},
\]
which is smooth and nonvanishing at $r=0$ only if $n\beta+m-1=0$.  Since $m\geq1$, this gives $\beta\leq0$.  Hence $\beta=0$.
\end{proof}

\begin{remark}\label{rem:contact}
Let $\eta=i_{r\partial_r}\widetilde\omega=\beta\,\u\cdot\d\s-\d v$, so that $\widetilde\omega=\frac{\d r}{r}\wedge\eta+\d\u\wedge\d\s$ and $\d\eta=\beta\,\d\u\wedge\d\s$.  For $\beta\neq0$,
\[
        \eta\wedge(\d\eta)^{n-1}=-\beta^{n-1}\,\d v\wedge\left(\d\u\wedge\d\s\right)^{n-1}\neq0 ,
\]
so the restriction of $\eta$ to the boundary $r=0$ is a contact form, and on the interior $\omega=\d\left(r^{\beta}\eta/\beta\right)$ is the symplectisation of $\eta/\beta$.  This is the contact structure obtained by reduction by the scaling symmetry \cite{LacombaIbort1988,BravettiJackmanSloan2023}.  For $\beta=0$, $\omega=\frac1r\,\d v\wedge\d r+\d\u\wedge\d\s$ is $b$-symplectic \cite{MirandaPlanas2022,BraddellEtAl2019}.  

On $\mcC$, which is invariant by Lemma~\ref{lem:mcgehee-equations},   the third equation of \eqref{eq:blownup} and equation \eqref{eq:vprime} give $\beta\,\u\cdot\s'-v'=\beta\,|\u|^2-v'=0$, so the flow on $\mcC$ is tangent to $\ker\eta$.
For $\beta\neq0$ this is the monotonicity of $v$.  For $\beta=0$, $v$ is constant on $\mcC$, and on each leaf $\{v=\text{const}\}$ the last two equations of \eqref{eq:blownup} are Hamilton's equations for $\frac12|\u|^2-U(\s)$ on $T^*{\mathcal S}^{n-1}$.
\end{remark}

\begin{remark}\label{rem:blowup}
Proposition~\ref{prop:mcgehee-isotropic} depends on the choice of McGehee's variables.  In the variables $(r,\,rp_r,\,\s,\,\p_\s)$,
\[
        \theta=rp_r\,\frac{\d r}{r}+\p_\s\cdot\d\s,
        \qquad
        \omega=\d(rp_r)\wedge\frac{\d r}{r}+\d\p_\s\wedge\d\s ,
\]
which is a $b$-symplectic form on $[0,\infty)\times\R\times T^*\mathcal O$.  For the Kepler problem, with $|\q|^{1/2}$ in place of $r$, this form is compared with the $b$-Darboux normal form in \cite{BravettiJackmanSloan2023}.  In these variables
\[
        \tfrac12\left((rp_r)^2+|\p_\s|^2\right)=r^2H+r^{2\beta}U(\s),
\]
so for $\beta>0$ the closure of every energy level meets the boundary $r=0$ only in $\{rp_r=0,\ \p_\s=0\}$, and for $\beta<0$ it does not meet it.  Thus for $\beta\neq0$ this boundary carries no collision manifold.  For $\beta=0$, $v=rp_r$ and $\u=\p_\s$, and the two blow-ups coincide.
\end{remark}

\medskip
\begin{remark}\label{rem:log}
The  logarithmic central force problem  is given by a Hamiltonian of the form 
\[
H=\frac12|\p|^2+\ln|\q|\,,\,\,\,\,\,\,\,\,\,\,\q \in \mathbb{R}^2\setminus\{{\bf 0}\}\,.
\] 
In \cite{SahaStoica2023} the authors apply the  blow-up 
 \[
 \q=r\,e^{-1/r^2}\s, \,\,\,\,\,\,\p=-r^{-1}e^{w}(\cos\psi,\sin\psi)
 \]
with $\s=\left(\cos(\varphi+\psi),\sin(\varphi+\psi)\right)$.  In these variables
\begin{align*}
        \omega&=e^{-1/r^2}\,\Theta,\qquad\text{where}\\
        \Theta&=e^{w}\Big[\frac{\d r}{r^3}\wedge\left((r^2+2)\cos\varphi\,\d w-r^2\sin\varphi\,\d\varphi+2\sin\varphi\,\d\psi\right)\\
        &\qquad\quad+\sin\varphi\,\d w\wedge(\d\varphi+\d\psi)+\cos\varphi\,\d\varphi\wedge\d\psi\Big].
\end{align*}
Thus $\Theta$ is a $b^3$-form, nondegenerate up to $r=0$, with $\d\Theta=-2\,\frac{\d r}{r^3}\wedge\Theta$.  The form $\omega$ vanishes to infinite order at $r=0$, so every admissible conformal factor in the sense of Definition~\ref{def:admissible} gives the zero limit, and Proposition~\ref{prop:scaling-obstruction} and Definition~\ref{def:isotropic} do not apply.  The contraction of $\Theta$ with $r^3\partial_r$ restricts to the boundary $r=0$ as the contact form $2e^{w}(\cos\varphi\,\d w+\sin\varphi\,\d\psi)$, and the flow of \cite{SahaStoica2023} on the boundary is tangent to its kernel.
\end{remark}

\subsection{Total collision in the planar three-body problem}\label{subsec:three-body}

We now consider total collision in the planar three-body problem with Newtonian interaction, for which $\beta=\frac12$.  In \cite{MoeckelMontgomery2013} the problem is reduced by translations and rotations, and the three binary collisions are regularised by Levi-Civita transformations, through partially symplectic maps.

At zero angular momentum, the only case in which total collision occurs, the result is a Hamiltonian system on $T^*(0,\infty)\times T^*{\mathcal S}^2$ with its canonical symplectic form, where $r$ is the mass norm of the configuration and ${\mathcal S}^2$ is the regularised shape sphere.  Following Section~6.3.3 in \cite{MoeckelMontgomery2013}, the regularised Hamiltonian at zero angular momentum  is
\begin{equation}\label{eq:H3}
        \widetilde H=\frac12\,\chi(\s)\,p_r^2+\frac{|\p_\s|^2}{2r^{2}}-\frac{\widetilde U(\s)}{r}-h\,\chi(\s),
\end{equation}
where $|\p_\s|$ is the norm given by a Riemannian metric on ${\mathcal S}^2$, $\widetilde U$ is smooth and positive, and $\chi\geq0$ is smooth and vanishes exactly at the six points of ${\mathcal S}^2$ over the three binary collisions.   (In the notation of \cite{MoeckelMontgomery2013}, $\s=c$, $\p_\s=\gamma$, $\chi=\tau$, $\widetilde U=W$, and $|\p_\s|^2=2\,(\tau/\lambda)\,|\gamma|^2$, where $\tau/\lambda$ is a smooth positive function on the sphere, so that the metric is a positive multiple of the round one).

\medskip

Denoting by $H$ the Hamiltonian of the reduced three-body problem written in the regularised variables, $\widetilde H=\chi\,(H-h)$, so that on the level $\widetilde H=0$ the vector field of $\widetilde H$ is $\chi$ times that of $H$.  Thus the solutions of the three-body problem with energy $h$ are the solutions of \eqref{eq:H3} on the level $\widetilde H=0$, with time $\tilde t$ given by $\d t=\chi\,\d\tilde t$.  The Hamiltonian \eqref{eq:H3} is smooth at the binary collisions, but remains singular at the triple collision $r=0$.

As in \cite[Section~7]{MoeckelMontgomery2013}, we apply McGehee's rescaling \eqref{eq:mcg} with $\beta=\frac12$, that is, with the factor $r^{3/2}$ in the time change:
%
\begin{equation}\label{eq:mcg3}
        v=r^{1/2}p_r,
        \qquad
        \u=r^{-1/2}\p_\s,
        \qquad
        \d\tilde t=r^{3/2}\,\d\tau .
\end{equation}
Thus we have:

\begin{lemma}\label{lem:three-body-equations}

In the variables $(r,v,\s,\u)$ and the time $\tau$, Hamilton's equations for \eqref{eq:H3} read
\begin{equation}\label{eq:blownup3}
\begin{aligned}
        r'&=\chi(\s)\,r\,v, \\
        v'&=\frac12\,\chi(\s)\,v^2+|\u|^2-\widetilde U(\s), \\
        \s'&=\partial_\u\left(\tfrac12|\u|^2\right), \\
        \u'&=-\frac12\,\chi(\s)\,v\,\u-\partial_\s\left(\tfrac12|\u|^2+\tfrac12\,\chi(\s)\,v^2-\widetilde U(\s)\right)+r\,h\,\partial_\s\chi(\s),
\end{aligned}
\end{equation}
where prime denotes differentiation with respect to $\tau$.  Denoting $\widetilde F:=r\widetilde H$, that is,
\begin{equation}\label{eq:energy3}
        \widetilde F=\frac12\,\chi(\s)\,v^2+\frac12\,|\u|^2-\widetilde U(\s)-r\,h\,\chi(\s),
\end{equation}
we have $\widetilde F'=\chi(\s)\,v\,\widetilde F$ along \eqref{eq:blownup3}.  Moreover, the level $\widetilde F=0$, the boundary $r=0$ and the triple collision manifold
\[
        \mcC_3=\left\{r=0,\ \tfrac12\,\chi(\s)\,v^2+\tfrac12\,|\u|^2=\widetilde U(\s)\right\}
\]
are invariant.
\end{lemma}

\begin{proof}
Denote by a dot the derivative with respect to $\tilde t$.  The canonical Hamilton's equations for \eqref{eq:H3} are
\[
\begin{aligned}
        &\dot r=\chi\,p_r,\qquad
        &&\dot p_r=\frac{|\p_\s|^2}{r^3}-\frac{\widetilde U}{r^2},\\
        &\dot\s=\frac{1}{r^2}\,\partial_{\p_\s}\left(\tfrac12|\p_\s|^2\right),\qquad
        &&\dot \p_\s=-\frac{1}{r^2}\,\partial_\s\left(\tfrac12|\p_\s|^2\right)-\frac12\,p_r^2\,\partial_\s\chi+\frac{1}{r}\,\partial_\s\widetilde U+h\,\partial_\s\chi .
\end{aligned}
\]
Since $\frac12|\p_\s|^2$ is quadratic in $\p_\s$, the substitution $\p_\s=r^{1/2}\u$ gives $\partial_{\p_\s}\left(\tfrac12|\p_\s|^2\right)=r^{1/2}\,\partial_\u\left(\tfrac12|\u|^2\right)$ and $\partial_\s\left(\tfrac12|\p_\s|^2\right)=r\,\partial_\s\left(\tfrac12|\u|^2\right)$.  We also use $r\,p_r^2=v^2$, $|\p_\s|^2=r\,|\u|^2$ and $p_r\,\p_\s=v\,\u$.  Differentiating $v=r^{1/2}p_r$ and $\u=r^{-1/2}\p_\s$ with respect to $\tilde t$ and multiplying by $r^{3/2}=\d\tilde t/\d\tau$, we obtain
\[
\begin{aligned}
        r'&=r^{3/2}\,\chi\,p_r=\chi\,r\,v,
        \qquad
        \s'=r^{3/2}\,r^{-2}\,r^{1/2}\,\partial_\u\left(\tfrac12|\u|^2\right)=\partial_\u\left(\tfrac12|\u|^2\right),\\
        v'&=\frac12\,r\,\chi\,p_r^2+r^2\,\dot p_r
        =\frac12\,\chi\,v^2+|\u|^2-\widetilde U,\\
        \u'&=-\frac12\,\chi\,p_r\,\p_\s+r\,\dot\p_\s
        =-\frac12\,\chi\,v\,\u-\partial_\s\left(\tfrac12|\u|^2\right)-\frac12\,v^2\,\partial_\s\chi+\partial_\s\widetilde U+r\,h\,\partial_\s\chi ,
\end{aligned}
\]
which is \eqref{eq:blownup3}. (These are the equations of \cite[Example~8, equations~(95) and~(96)]{MoeckelMontgomery2013} with McGehee's factor ($\nu=1$), zero angular momentum ($\tilde\mu=0$) and $|c|=1$.)

Multiplying \eqref{eq:H3} by $r$ and using the same relations gives $r\widetilde H=\widetilde F$.  Since $\widetilde H$ is constant along its flow, $\widetilde F'=r'\,\widetilde H=\chi\,v\,r\widetilde H=\chi\,v\,\widetilde F$.  Hence $\{\widetilde F=0\}$ is invariant, $\{r=0\}$ is invariant because $r'=\chi rv$, and $\mcC_3$ is their intersection.  The right-hand sides of \eqref{eq:blownup3} and the function $\widetilde F$ are smooth at the points over binary collisions, where $\chi=0$.
\end{proof}

\medskip

\begin{proposition}[Total collision in the planar three-body problem]\label{prop:three-body}
In the variables $(r,v,\s,\u)$ the canonical form is given by \eqref{eq:omega-tilde} with $\beta=\frac12$.  Consequently:
\begin{enumerate}[label=(\alph*)]

\item at every point of the boundary $r=0$, including the points over binary collisions, $\widetilde\omega$ is a nondegenerate $b$-form, $\omega$ is isotropic and not conformally regularisable, and no admissible conformal factor makes $\omega$ a $b^m$-symplectic form for any integer $m\geq1$;

\item the restriction of $\eta=\frac12\,\u\cdot\d\s-\d v$ to $r=0$ is a contact form;

\item along \eqref{eq:blownup3},
\[
        \frac12\,\u\cdot\s'-v'=-\widetilde F-r\,h\,\chi(\s),
\]
so the flow on $\mcC_3$ is tangent to $\ker\eta$; moreover $v'=\frac12|\u|^2\geq0$ on $\mcC_3$.
\end{enumerate}
\end{proposition}

\begin{proof}
Since $p_r=r^{-1/2}v$ and $\p_\s=r^{1/2}\u$, the canonical one-form is $\theta=r^{-1/2}v\,\d r+r^{1/2}\,\u\cdot\d\s$, so $\omega=\d\theta$ is given by \eqref{eq:omega-tilde} with $\beta=\frac12$.  The proofs of Proposition~\ref{prop:mcgehee-isotropic} and of the contact property in Remark~\ref{rem:contact} use only \eqref{eq:omega-tilde} and the fact that $\d\u\wedge\d\s$ is the canonical symplectic form of $T^*{\mathcal S}^{n-1}$ with $n=3$.  They apply at every point of $r=0$, since the binary collisions correspond to ordinary points of ${\mathcal S}^2$.  This proves (a) and (b).

For (c), Euler's relation for the quadratic function $\frac12|\u|^2$ gives $\u\cdot\s'=|\u|^2$, so by \eqref{eq:blownup3} and \eqref{eq:energy3}
\[
        \frac12\,\u\cdot\s'-v'=-\left(\frac12\,\chi(\s)\,v^2+\frac12\,|\u|^2-\widetilde U(\s)\right)=-\widetilde F-r\,h\,\chi(\s).
\]
On $\mcC_3$ both $\widetilde F$ and $r$ vanish, and $\frac12\,\chi(\s)\,v^2=\widetilde U(\s)-\frac12|\u|^2$, which gives $v'=\frac12|\u|^2$.
\end{proof}

\subsection{Partial collisions}\label{subsec:partial}

We now consider singularities in which part of the system does not take part in the collapse.  Specifically, we consider the $N$-body problem with pair interactions homogeneous of degree $-\alpha$.  We write the Hamiltonian in mass-weighted Jacobi coordinates $(\q,\p,\x,\y)$, where $\q\in\R^n$ collects the internal coordinates of one or several clusters, each colliding internally, $\p\in\R^n$ are the conjugate momenta, and $(\x,\y)\in\R^{k}\times\R^{k}$, $k\geq1$, are the coordinates and momenta of the rest of the system, which stays regular.  The Hamiltonian reads
\begin{equation}\label{eq:H-partial}
        H=\frac12|\p|^2-|\q|^{-\alpha}\,U\!\left(\frac{\q}{|\q|}\right)+G(\x,\y)-W(\q,\x,\y),
\end{equation}
where $U$ is smooth and positive on an open subset $\mathcal O$ of the unit sphere ${\mathcal S}^{n-1}$, $G$ is smooth, and $W$ is smooth in a neighbourhood of $\q=\0$.  The set $\q=\0$ is the collision of the clusters, the function $W$ collects the interactions between bodies of different clusters and with the remaining bodies, and the points of ${\mathcal S}^{n-1}$ outside $\mathcal O$ correspond to collisions of only part of the colliding bodies.  The coordinates are canonical and so the symplectic form is
\[
        \omega=\d\p\wedge\d\q+\d\y\wedge\d\x\,.
\]

\medskip
Following the McGehee blow-up as developed in \cite{DuignanQuaschner2026}, we write $\q=r\s$ with $r=|\q|$ and $\s\in\mathcal O$, introduce McGehee variables \eqref{eq:mcg} for the $(\q,\p)$-part, and keep $(\x,\y)$ unchanged.  The blown-up space has dimension $2(n+k)$ and boundary $r=0$, and the collar is given by the vector field $X=\partial_r$ in the variables $(r,v,\s,\u,\x,\y)$.
Since $\omega$ does not involve $U$, $G$ or $W$, in the new coordinates equation \eqref{eq:omega-tilde}  reads
\begin{equation}\label{eq:omega-sim}
        \omega=r^{\beta}\left[\frac{\d r}{r}\wedge\left(\beta\,\u\cdot\d\s-\d v\right)+\d\u\wedge\d\s\right]+\d\y\wedge\d\x\,\,\,\, \text{for} \,\,\,n\geq2
\end{equation}
and 
\begin{equation}\label{eq:omega-sim_2}
\omega=r^{\beta-1}\,\d v\wedge\d r+\d\y\wedge\d\x\,\,\,\, \text{for} \,\,\,n=1\,.
\end{equation}

\begin{proposition}[Partial collisions]\label{prop:simultaneous}
Let $k\geq1$.
\begin{enumerate}[label=(\alph*)]
\item If $n\geq2$ and $\beta\neq0$, then $\omega$ is anisotropic at $r=0$ in the sense of Definition~\ref{def:isotropic}.  Consequently, for every integer $m\geq0$ there is no admissible conformal factor $f$ such that $f\omega$ extends to the blown-up space as a $b^m$-form that is nondegenerate at every point of $r=0$.
\item If $n=1$, then $\omega$ is isotropic at $r=0$.  For an integer $m\geq1$, there is an admissible conformal factor $f$ such that $f\omega$ extends to the blown-up space as a $b^m$-symplectic form if and only if $\beta=1-m$, and then $\omega$ itself is $b^m$-symplectic.
\item If $\beta=0$, then $\omega$ is $b$-symplectic.
\end{enumerate}
\end{proposition}

\begin{proof}
(a) The terms containing $\d r$ vanish on $\ker\d r$, so the tangential part of $\omega$ is
\[
        \sigma=r^{\beta}\,\d\u\wedge\d\s+\d\y\wedge\d\x .
\]
If $\beta>0$, then $\sigma$ tends to $\sigma_0=\d\y\wedge\d\x$, of rank $2k$.  If $\beta<0$, then $r^{-\beta}\sigma=\d\u\wedge\d\s+r^{-\beta}\,\d\y\wedge\d\x$ tends to $\sigma_0=\d\u\wedge\d\s$, of rank $2n-2$.  Since $n\geq2$ and $k\geq1$, both ranks are less than $2(n+k)-2$, the rank in Definition~\ref{def:isotropic} for the blown-up space of dimension $2(n+k)$.  Thus $\omega$ is anisotropic, and Proposition~\ref{prop:bm-obstruction} applies.

(b) For $n=1$ the tangential part of $\omega$ is $\d\y\wedge\d\x$, of rank $2k=2(1+k)-2$, so $\omega$ is isotropic.  If $f\omega$ extends as a $b^m$-symplectic form, then $f\omega$ is closed on $\{0<r<\lam_0\}$, and since the blown-up space has dimension $2(1+k)\geq4$, $f$ is a constant $c\neq0$ by Proposition~\ref{prop:conformal-symplectic}, applied to $1/|f|$.  In the $b^m$-coframe $\d r/r^m$, $\d v$, $\d\x$, $\d\y$,
\[
        c\,\omega=c\,r^{m-1+\beta}\,\frac{\d r}{r^m}\wedge(-\d v)+c\,\d\y\wedge\d\x ,
\]
whose coefficient $c\,r^{m-1+\beta}$ is smooth and nonvanishing at $r=0$ if and only if $m=1-\beta$.  Conversely, if $\beta=1-m$, then $\omega=\frac{\d r}{r^m}\wedge(-\d v)+\d\y\wedge\d\x$ is closed and nondegenerate as a $b^m$-form.

(c) For $\beta=0$, $\omega$ is the sum of the $b$-symplectic form of Proposition~\ref{prop:mcgehee-isotropic}(a) in the variables $(r,v,\s,\u)$ and the symplectic form $\d\y\wedge\d\x$; for $n=1$ this is (b) with $m=1$.
\end{proof}

The difference with the total collision of Section~\ref{subsec:central} is that the obstruction is now pointwise: the two blocks of the tangential part carry the weights $r^{\beta}$ and $r^0$, and no scalar factor can balance them.  Also, $\eta$ does not restrict to a contact form on the boundary: for $\eta=\beta\,\u\cdot\d\s-\d v$ one has $\d\eta=\beta\,\d\u\wedge\d\s$, of rank $2n-2$, so $\eta\wedge(\d\eta)^{n+k-1}=0$.

\medskip
\begin{remark}\label{rem:partial-examples}
In the $N$-body problem in $\R^d$, Proposition~\ref{prop:simultaneous} applies, with $k\geq1$, to a binary collision with $N\geq3$, for which $n=d$; to a triple collision with $N\geq4$, for which $n=2d$; and to two simultaneous binary collisions with $N\geq4$, for which $n=2d$.  Hence, in McGehee's variables, these singularities are anisotropic for $\beta\neq0$, except for a binary collision on a line, $d=1$, which is isotropic.  For Newtonian gravity, $\beta=\frac12$, the latter is not conformally $b^m$-symplectic for any $m\geq 1$.
\end{remark}

\medskip

\begin{remark}\label{rem:four-body}
The simplest instance of two simultaneous binary collisions is the collinear four-body problem, with positions $x_1<x_2<x_3<x_4$ and the pairs $(1,2)$ and $(3,4)$ colliding at the same instant while staying apart from each other; for Newtonian gravity see \cite{MartinezSimo1999,DuignanDullin2021}.  Here $\q=(q_1,q_2)$ are the separations $x_2-x_1$ and $x_4-x_3$ of the two pairs, rescaled by the square roots of their reduced masses, $\x=x$ is the distance between the centres of mass of the two pairs, and $\y=y$ is its conjugate momentum.  The Hamiltonian is \eqref{eq:H-partial} with $n=2$, $k=1$, $G=y^2/(2\mu)$, where $\mu$ is the reduced mass of the two pairs,
\[
        U(\s)=\frac{c_1}{s_1^{\alpha}}+\frac{c_2}{s_2^{\alpha}},
        \qquad \s=(s_1,s_2)\in\mathcal O=\{s_1>0,\ s_2>0\},
\]
with constants $c_1,c_2>0$ depending only on the masses, and $W$ is the sum of the four interactions between bodies of different pairs, which is smooth near $\q=\0$ as long as $x>0$.  The ends of the arc $\mathcal O$ correspond to a collision of one pair alone.  By Proposition~\ref{prop:simultaneous}, $\omega$ is anisotropic for $\beta\neq0$ and $b$-symplectic for $\beta=0$.
\end{remark}

\medskip

\begin{remark}\label{rem:partial-scaling}
The scaling generated by $r\partial_r$ acts on $(\q,\p)$ and leaves $(\x,\y)$ unchanged.  Since the coefficients of the bracket in \eqref{eq:omega-sim} do not depend on $r$ in the $b$-coframe,
\[
        \Lie_{r\partial_r}\,\omega=\beta\left(\omega-\d\y\wedge\d\x\right),
\]
which is a multiple of $\omega$ only for $\beta=0$; for the total collision of Section~\ref{subsec:central}, $\Lie_{r\partial_r}\,\omega=\beta\,\omega$.  As in Remark~\ref{rem:blowup}, Proposition~\ref{prop:simultaneous} depends on the choice of McGehee's variables: in the variables $(r,\,rp_r,\,\s,\,\p_\s,\,\x,\,\y)$,
\[
        \omega=\d(rp_r)\wedge\frac{\d r}{r}+\d\p_\s\wedge\d\s+\d\y\wedge\d\x
\]
is $b$-symplectic for every $\beta<1$.
\end{remark}

\section{Concluding remarks}\label{sec:conclusions}

Singular regularisations in Hamiltonian systems involve two  questions:
\begin{enumerate}[label=(\roman*)]
\item Does the vector field extend to a singular boundary after blow-up and time rescaling?
\item Does the symplectic form extend to that boundary, after a conformal rescaling, as a symplectic or as a $b^m$-symplectic form?
\end{enumerate}

For the McGehee blow-up of a collision the answer to the first question is yes, by construction. 
In Proposition~\ref{prop:scaling-obstruction} we provide the reason for a negative answer to the second question. Its  hypothesis is met already by $b^m$-symplectic forms, in which the normal and the tangential components carry different weights, and this difference is absorbed by passing to $b^m$-forms.  At a total collision of a central force problem or of the planar three-body problem, the canonical form is $r^{\beta}$ times a nondegenerate $b$-form which is closed only for $\beta=0$.  For $\beta\neq0$ the collision manifold carries a contact structure and there is no $b^m$-symplectic rescaling; for $\beta=0$ the canonical form is $b$-symplectic.  At a partial collision the tangential part of the canonical form carries two weights, one for the colliding part and one for the part that stays regular, and Proposition~\ref{prop:bm-obstruction} excludes every nondegenerate $b^m$-form after a conformal rescaling, unless the colliding part has a single degree of freedom.

\medskip
Possible extensions include a more systematic treatment of anisotropic singular symplectic structures near collision boundaries, complementing the isotropic $b^m$-symplectic and folded models of \cite{GuilleminMirandaPires2014,Scott2016,DelshamsKiesenhoferMiranda2017}.

\section*{Acknowledgements}
The author was supported by an NSERC Discovery Grant. The author used the LLMs ChatGPT (OpenAI), Claude (Anthropic)  to assist with copy-editing and with restructuring parts of the exposition. All mathematical ideas, results, and verifications are due to the author, who retains full responsibility for the correctness and originality of the manuscript.

\end{document}